\documentclass[12pt]{amsart}%
\usepackage{amsmath}
\usepackage[margin=1in]{geometry}
\usepackage{graphicx}
\usepackage{amsfonts}
\usepackage{amssymb}%
\usepackage[utf8]{inputenc}
\usepackage{comment}
\usepackage{hyperref}

\newtheorem{theorem}{Theorem}[section]
\theoremstyle{plain}

\newtheorem{claim}{Claim}

\newtheorem{lemma}{Lemma}[section]

\newtheorem{proposition}{Proposition}[section]

\numberwithin{equation}{section}
\theoremstyle{definition}

\theoremstyle{remark}
\newtheorem{remark}{Remark}[section]

\begin{document}
\title[Hessian Estimates]{
Hessian Estimates for Lagrangian mean curvature equation}
\author{Arunima Bhattacharya}
\address{Department of Mathematics\\
University of Washington, Seattle, WA 98195, U.S.A.}
\email{arunimab@uw.edu}

\begin{abstract}
 In this paper, we derive a priori interior Hessian estimates for the Lagrangian mean curvature equation if the Lagrangian phase is supercritical and has bounded second derivatives.

\end{abstract}
\maketitle

\section{ Introduction}
In this paper, we study a priori interior Hessian estimates in all dimensions for the Lagrangian mean curvature equation
\begin{equation}
    F(D^{2}u)=\sum _{i=1}^{n}\arctan \lambda_{i}=\psi(x) \label{s}
\end{equation} 
under the assumption that $|\psi|\geq (n-2)\frac{\pi}{2}+\delta$ where $\delta>0$, and $\psi$ has bounded second derivatives.
Here $\lambda_i$'s are the eigenvalues of the Hessian matrix $D^2u$ and then the phase $\psi$  becomes a potential for the mean curvature of the Lagrangian submanifold $(x,Du(x))\subseteq \mathbb{R}^n\times\mathbb{R}^n$. When the phase $\psi$ is constant, denoted by $c$, $u$ solves the special Lagrangian equation 
 \begin{equation}
\sum _{i=1}^{n}\arctan \lambda_{i}=c \label{s1}
\end{equation} or equivalently,
\[ \cos c \sum_{1\leq 2k+1\leq n} (-1)^k\sigma_{2k+1}-\sin c \sum_{0\leq 2k\leq n} (-1)^k\sigma_{2k}=0.
\]
 Equation (\ref{s1}) originates in the special Lagrangian geometry by Harvey-Lawson \cite{HL}. The Lagrangian graph $(x,Du(x)) \subset \mathbb{R}^n\times\mathbb{R}^n$ is called special
when the argument of the complex number $(1+i\lambda_1)...(1+i\lambda_n)$
or the phase $\psi$ is constant, and it is special if and only if $(x,Du(x))$ is a
(volume minimizing) minimal surface in $(\mathbb{R}^n\times\mathbb{R}^n,dx^2+dy^2)$ \cite{HL}.

A dual form of (\ref{s1}) is the Monge-Amp\'ere equation
\begin{equation*}
    \sum_{i=1}^n \ln\lambda_i=c.
\end{equation*}
This is the potential equation for special Lagrangian submanifolds in $(\mathbb {R}^n\times \mathbb {R}^n, dxdy)$ as interpreted in \cite{Hi}. The gradient graph $(x,Du(x))$ is volume maximizing in this pseudo-Euclidean space as shown by Warren \cite{W}. In the 1980s, Mealy \cite{Me} showed that an equivalent algebraic form of the above equation is the potential equation for his volume maximizing special Lagrangian submanifolds in $(\mathbb R^n\times \mathbb R^n, dx^2-dy^2)$.

The arctangent operator or the logarithmic operator is concave if $u$ is convex, or if the Hessian of $u$ has a lower bound $\lambda\geq 0$. Certain concavity properties of the arctangent operator are still preserved for saddle $u$. The concavity of the arctangent operator in (\ref{s}) depends on the range of the Lagrangian phase. The phase $(n-2)\frac{\pi}{2}$ is called critical because the level set $\{ \lambda \in \mathbb{R}^n \vert \lambda$ satisfying $ (\ref{s})\}$ is convex only when $|\psi|\geq (n-2)\frac{\pi}{2}$ \cite[Lemma 2.2]{YY}. The concavity of the level set is evident for $|\psi|\geq (n-1)\frac{\pi}{2}$ since that implies $\lambda>0$ and then $F$ is concave. For solutions of (\ref{s1}) with critical and supercritical phases $|\psi|\geq (n-2)\frac{\pi}{2}$, Hessian estimates have been obtained by Warren-Yuan \cite{WY9,WY} and Wang-Yuan \cite{WaY}. 
If the phase is subcritical $|\psi|<(n-2)\frac{\pi}{2}$, solutions of (\ref{s1}) fail to have interior estimates as shown in examples of Nadirashvili-Vl\u{a}du\c{t} \cite{NV} and Wang-Yuan \cite{WangY}.

Our main results in this paper are the following:
\begin{theorem}\label{main1}
Let $u$ be a $C^4$ solution of (\ref{s}) on $B_{R}(0)\subset \mathbb{R}^{n}$ where $\psi\in C^{1,1}(B_{R}) $, $|\psi|\geq (n-2)\frac{\pi}{2}+\delta$. Then we have 
\begin{equation}
    |D^2u(0)|\leq C_1\exp [C_2\max_{B_R(0)}|Du|^{2n-2}/R^{2n-2}] \label{H}
\end{equation}
where $C_1$ and $C_2$ are positive constants depending on $||\psi||_{C^{1,1}(B_{R})}$, $n$, and $\delta$.
\end{theorem}

In order to link the dependence of Hessian estimates
to the potential $u$ itself, we have the following gradient estimate.
\begin{theorem}\label{main2}
Let $u$ be a $C^3$ solution of (\ref{s}) on $B_{3R}(0)\subset \mathbb{R}^{n}$ where $\psi\in C^{1,1}(B_{3R}) $, $|\psi|\geq (n-2)\frac{\pi}{2}+\delta$. Then we have 
\begin{equation}
    \max_{B_{R}(0)}|Du|\leq C_3 osc_{B_{3R}(0)}\frac{u}{R}+C_4(n) \label{g1}
\end{equation}
where $C_3$ is a positive constant depending on $||\psi||_{C^{1}(B_{3R})}$, $n$, and $\delta$.
\end{theorem}
\begin{remark}
For the constant critical and supercritical phase equation (\ref{s1}), the constants $C_3$ and $C_4$ are only dimensional \cite[Theorem 1.2]{WY}. The constant $C_4$ was further reduced to $0$ in Yuan's unpublished 2015 notes on special Lagrangian equations.
\end{remark}
\begin{remark}
One application of the above estimates is the regularity (analyticity) of
$C^0$ viscosity solutions of (\ref{s}) where $|\psi|\geq (n-2)\frac{\pi}{2}+\delta$; solutions of the Dirichlet problem of (\ref{s}) with continuous boundary data enjoy interior regularity, as shown in \cite[Theorem 1.1]{AB1}. Another application is the regularity of convex $C^0$ viscosity solutions of (\ref{s}) where $\psi\in C^{2,\alpha}(B_1)$, as shown in \cite[Theorem 1.1]{BS}.
\end{remark}

\begin{remark} For Theorem \ref{main1}, an assumption weaker than $C^{1}$ on $\psi$ leads to counterexamples. For example, in two dimensions, we consider a boundary value problem of (\ref{s}) on the unit ball $B_1(0)$ where the phase is in $C^{\alpha}$ with $\alpha\in (0,1)$: $\psi(x)=\frac{\pi}{2}-\arctan (\alpha^{-1}|x|^{1-\alpha})$ and $u(x)=\int_{0}^{|x|}t^{\alpha}dt$ on $\partial B_1$. Now if Hessian estimates hold good for a H\"older continuous phase then by a smooth approximation of the phase and boundary data we would find a solution that is $C^{2,\alpha}$ in the interior of $B_1$. However, this boundary value problem admits a non $C^2$ unique viscosity solution $u$ with gradient $D u=|x|^{\alpha-1}x$, thereby proving a contradiction. 
\end{remark}
\begin{remark}
The existence of interior estimates for solutions of (\ref{s}) with critical and supercritical phase where $\psi\in C^{1,\varepsilon_{0}}$, or even $|\psi|\geq (n-2)\frac{\pi}{2}$ where $\psi\in C^{1,1}$ are still open questions. Again if the phase is subcritical then even for the constant phase equation (\ref{s1}), singular $C^{1,\varepsilon}$ viscosity solutions were constructed in \cite{NV, WangY}.
\end{remark}
\medskip

For the two dimensional case, Heinz \cite{H} derived a Hessian bound for solutions of the Monge-Ampère type equation including (\ref{s1}); Pogorelov \cite{P1} derived Hessian estimates for solutions of these equations including (\ref{s1}) with $|\psi|\geq\frac{\pi}{2}$. Later Pogorelov \cite{P2} constructed his famous
counterexamples for the three dimensional Monge-Ampère equation $\sigma_3(D^2u) = \det(D^2u) = 1$, which also
serve as counterexamples for cubic and higher order symmetric $\sigma_k$ equations (see \cite{U1}). Gregori \cite{Gg} extended Heinz’s estimate to a gradient
bound in terms of the heights of the two dimensional minimal surfaces, and for graphs with non-zero mean curvature an additional requirement on the length of the mean curvature vector was assumed. Hessian estimates for solutions to Monge-Ampère equations and the $\sigma_k$
equations for $k\geq 2$ were established by Pogorelov \cite{P2} and Chou-Wang \cite{ChW} under certain strict convexity constraints. For the three dimensional case, Trudinger \cite{T2}, Urbas \cite{U2,U3}, and Bao-Chen
\cite{BC} obtained pointwise Hessian estimates in terms of certain integrals of
the Hessian, for $\sigma_k$ equations and the special Lagrangian equation (\ref{s1}) with $c=\pi$ respectively.  Bao-Chen-Guan-Ji \cite{BCGJ} obtained pointwise Hessian estimates for strictly convex solutions to quotient equations $\sigma_n=\sigma_k$ in terms of certain
integrals of the Hessian. 
  Recently, along the integral way, Qiu \cite{Q} proved Hessian estimates for solutions of the three dimensional quadratic
Hessian equation with a $C^{1,1}$ variable right hand side. Hessian estimates for
convex solutions of general quadratic Hessian equations were obtained via a new pointwise approach by Guan-Qiu \cite{GQ}. For convex viscosity solutions of (\ref{s1}) Hessian estimates have been obtained by Chen-Warren-Yuan \cite{WYJ} and Chen-Shankar-Yuan \cite{RYJ}. Hessian estimates for semiconvex smooth solutions and almost convex viscosity solutions of $\sigma_2(D^2u)=1$ were recently established by Shankar-Yuan in \cite{RY1} and \cite{RYY} respectively.

Our proof of the Hessian estimates goes as follows: we first bound the Hessian of $u$ by its integral followed by
an integral of its gradient, then by the volume of the Lagrangian graph, and lastly,
by the height of the Lagrangian graph, which is the gradient of the solution of (\ref{s}). The presence of $\psi(x)$ in the non-uniformly elliptic equation \eqref{s} presents unique challenges. One of the difficulties is the unavailability of harmonic co-ordinates $\Delta_g x=0$ since the Lagrangian graph $(x,Du(x))\subset \mathbb{R}^n\times\mathbb{R}^n$ is not a minimal surface. As a result, the linearized operator of (\ref{s}) at $u$ does not represent the Laplace-Beltrami operator, like in the constant phase case \cite{WaY,WY}.
Another hurdle is proving Jacobi inequalities, which require differentiating \eqref{s} twice. 
In the homogeneous case, one can differentiate \eqref{s1} and recover $D^3u$, but upon differentiating \eqref{s}, we end up with terms involving derivatives of $\psi$, which, a priori, could be large compared to the coefficients $DF(\lambda)$ of $D^3u$. We prove a Jacobi type inequality for $b=\ln\sqrt{1+\lambda_{\max}^2}$ where its Hessian is bounded below by its gradient and the $C^{1,1}$ norm of $\psi$. Applying a mean value inequality for $b$ and  certain Sobolev inequalities we estimate the integral of $b$ by a weighted volume of the non-minimal Lagrangian graph. By a conformality identity, the weighted volume element turns out to be a linear combination of the elementary symmetric functions of $D^2u$.
The linear combination poses yet another difficulty but we take advantage of the divergence type structure to bound
the weighted volume of the Lagrangian graph in terms of its height and the $C^{1,1}$ norm of $\psi$.

Through out this paper we assume $\psi\geq (n-2)\frac{\pi}{2}+\delta$ since
by symmetry $\psi\leq- (n-2)\frac{\pi}{2}-\delta$ can be treated similarly. This paper is divided into the following sections: in section two, we  introduce some notations and state some well known trigonometric inequalities satisfied by solutions of (\ref{s}), which will be used later in the proofs. In section three, we establish the gradient estimates, thereby proving Theorem \ref{main2}. In section four, we prove the pointwise and integral Jacobi inequality. In section five, we prove a mean value inequality for functions satisfying a Jacobi type inequality on submanifolds with high co-dimension, followed by the proof of Theorem \ref{main1}.\\

\section{Preliminaries}

\subsection{Notations} 
We introduce some notations that will be used in this paper.
The induced Riemannian metric on the Lagrangian submanifold $X=(x,Du(x))\subset \mathbb{R}^n\times\mathbb{R}^n$ is given by
\[g=I_n+(D^2u)^2 .
\]
We denote
 \begin{align*} 
    \partial_i=\frac{\partial}{\partial_{x_i}}\\
    \partial_{ij}=\frac{\partial^2}{\partial_{x_i}\partial_{x_j}}\\
    u_i=\partial_iu\\
    u_{ij}=\partial_{ij}u.
    \end{align*}
  Note that for the functions defined below, the subscripts on the left do not represent partial derivatives\begin{align*}
    b_k=(\ln\sqrt{1+\lambda_1^2}+...+\ln\sqrt{1+\lambda_k^2})/k\\
    h_{ijk}=\sqrt{g^{ii}}\sqrt{g^{jj}}\sqrt{g^{kk}}u_{ijk}\\
    g^{ii}=\frac{1}{1+\lambda_i^2}.
    \end{align*}
Here $(g^{ij})$ is the inverse of the matrix $g$ and $h_{ijk}$ denotes the second fundamental form when the Hessian of $u$ is diagonalized.
The volume form, gradient, and inner product with respect to the metric $g$ are given by
\begin{align*}
    dv_g=\sqrt{\det g}dx\\
    \nabla_g v=g^{ij}v_iX_j\\
    \langle\nabla_gv,\nabla_g w\rangle_g =g^{ij}v_iw_j\\
    |\nabla_gv|^2=\langle\nabla_gv,\nabla_g v\rangle_g.
\end{align*}

\subsection{Laplace-Beltrami operator and mean curvature formula} 
  Taking variations of the energy functional $\int |\nabla_g v|^2 dv_g$ with respect to $v$, one has the Laplace-Beltrami operator of the metric $g$:
\begin{align}
\Delta_g  =\frac{1}{\sqrt{ g}}\partial_i(\sqrt{ g}g^{ij}\partial_j )
=g^{ij}\partial_{ij}+\frac{1}{\sqrt{g}}\partial_i(\sqrt{g}g^{ij})\partial_j \label{2!}\\
=g^{ij}\partial_{ij}-g^{jp}\psi_q u_{pq} \partial_j. \nonumber
\end{align} The last equation follows from the following intrinsic and then extrinsic computations:

\begin{align}
    \frac{1}{\sqrt{g}}\partial_i(\sqrt{g}g^{ij})=\frac{1}{\sqrt{g}}\partial_i(\sqrt{g})g^{ij}+\partial_ig^{ij}=\frac{1}{2}(\partial_i \ln g)g^{ij}+\partial_kg^{kj}\nonumber\\
    =\frac{1}{2}g^{kl}\partial_i g_{kl}g^{ij}-g^{kl}\partial_k g_{lb}g^{bj}\nonumber\\
    =-g^{jp}g^{ab}u_{abq}u_{pq}
    =-g^{jp}\psi_q u_{pq}\label{lolz}
\end{align}
where the last equation follows from (\ref{111}) and (\ref{linearize}) below.
The first
derivative of the metric $g$ is given by 
\begin{align}
    \partial_i g_{ab}=\partial_i(\delta_{ab}+u_{ak}u_{kb})=u_{aik}u_{kb}+u_{bik}u_{ka}\overset{\text{at } x_0}{=}u_{abi}(\lambda_a+\lambda_b) \label{111}
    \end{align}
    assuming the Hessian of $u$ is diagonalized at $x_0$.
On taking the gradient of both sides of the Lagrangian mean curvature equation (\ref{s}), we get
\begin{equation}
\sum_{a,b=1}^{n}g^{ab}u_{jab}=\psi_j .\label{linearize}
\end{equation}

The coefficients, given by (\ref{lolz}), are in fact equal to the tangential part of the following decomposition of $X_{ij}$
where $X=(x,Du(x))$
\begin{align*}
     X_{ij}=(X_{ij})^T+(X_{ij})^N=\langle X_{ij},X_a\rangle g^{ab}X_b+II_{ij}\\
     =u_{ijk}u_{ka}g^{ab}X_b+II_{ij}\overset{\text{define }}=\Gamma^b_{ij}X_b+II_{ij}
\end{align*}
    where $\Gamma_{ij}^b$ is the Christoffel symbol.

On taking trace with respect to the metric $g$ and projecting to the tangential direction $X_l=(\partial_l, Du_l)$, we get
\[g^{ab}\Gamma^m_{ab}=g^{ab}u_{kab}u_{kl}g^{lm}=\psi_k u_{kl}g^{lm},
    \] 
which is the coefficient derived in (\ref{lolz}).
In turn, the normal part is
\[\vec{H}=g^{ab}II_{ab}=g^{ab}(\partial_{ab}X-\Gamma^m_{ab}\partial_m X)=\Delta_g X.
\]

     On the other hand, directly projecting to the normal direction  $J X_l$, we get the mean curvature vector of the Lagrangian submanifold $(x,Du(x))$
    \begin{equation}
    \vec{H}=g^{ab}\langle(-Du_{ab},0),-X_l\rangle g^{lm}JX_m=g^{ab}u_{lab}g^{lm}JX_m=\psi_l g^{lm}JX_m=J\nabla_g \psi \label{mean}
\end{equation}
where $\nabla_g$ is the gradient operator for the metric $g$ and $J$ is the complex structure, or the $\frac{\pi}{2}$ rotation matrix in  $\mathbb{R}^n\times \mathbb{R}^n$ and we used (\ref{linearize}) for the second  last equation. Note that the above formula for mean curvature of the Lagrangian submanifold $(x,Du(x))$ was originally found in \cite[(2.19)]{HL}.

\begin{remark}
When $\psi$ is constant, harmonic co-ordinates $\Delta_g x=0$ are available, which 
 reduces the Laplace-Beltrami operator on the minimal submanifold $\{(x,Du(x))|x\in B_R(0)\}$ to the linearized operator of (\ref{s1}) at $u$.
 Also, note that in this paper, by our assumption on $\psi$, $|H|$ is bounded. 
 \end{remark} 
Next we state the following Lemma.
\begin{lemma}\label{y1}
		Suppose that the ordered real numbers $\lambda_{1}\geq \lambda_{2}\geq...\geq \lambda_{n}$ satisfy (\ref{s}) with $\psi\geq (n-2)\frac{\pi}{2}$.
		Then we have \begin{enumerate}
			\item $\lambda_{1}\geq \lambda_{2}\geq...\geq \lambda_{n-1}>0, \lambda_{n-1}\geq |\lambda_{n}|$,
			\item $\lambda_{1}+(n-1)\lambda_{n}\geq 0$,
			\item $\sigma_{k}(\lambda_{1},...,\lambda_{n})\geq 0$ for all $1\leq k\leq n-1$ and $n\geq 2$,
			\item  if $\psi\geq (n-2)\frac{\pi}{2}+\delta$, then $D^2u\geq -\cot \delta I_n$.
			
		\end{enumerate}
		
	\end{lemma}
\begin{proof}
Properties (1), (2), and (3) follow from \cite[Lemma 2.2]{WaY}. Property (4) follows from \cite[Pg 1356]{YY}.
\end{proof}
\section{Gradient Estimates}
We prove Theorem \ref{main2}.
\begin{proof}
We assume $R=1$ by scaling $\frac{u(Rx)}{R^2}$.

We set $M=osc_{B_1}u$. W.l.o.g we assume $M>0$. Replacing $u$ by $u-\min_{B_1}u+M$ we now have 
\begin{equation}
M\leq u\leq 2M \label{q19}
\end{equation}
in $B_1$. We define 
\[w=\eta |Du|+Au^2\] with $\eta=1-|x|^2$ and 
\begin{equation}
A=\frac{n}{M}. \label{q29}
\end{equation}
If $w$ attains its supremum on
the boundary, then we are done. 
So we assume that $w$ attains its supremum at an interior point $x_0\in B_1$. We choose a co-ordinate system so that $D^2u$ is diagonalized at $x_0$. Let's assume that $u_n\geq \frac{|Du|}{\sqrt{n}}>0$ at $x_0$. For $1 \leq i\leq n$ we have
\begin{equation}
\frac{u_{i}u_{ii}}{|Du|}=|Du|_i=-\frac{\eta_i|Du|+2Auu_i}{\eta}. \label{g11}
\end{equation}
Observe that $u_{nn}<0$ since $A=\frac{n}{M}$. Since $\psi\geq(n-2)\frac{\pi}{2}+\delta$ we must have $\lambda_{min}=\lambda_n$ and $\lambda_k\geq |\lambda_n|$ by Lemma \ref{y1}. So we see that for $1\leq k\leq n-1$
\begin{equation}
    g^{nn}=\frac{1}{1+\lambda_n^{2}}\geq \frac{1}{1+\lambda_k^{2}}=g^{kk} \label{g12}
\end{equation}
and
\begin{equation}
    \frac{1}{g^{nn}}=1+\lambda_n^{2}<C(\delta) \label{k}.
\end{equation}

Next we define the following operator \[Lu=\sum_{i=1}^{n}g^{ii}u_{ii}.\] Note that this is the Laplace-Beltrami operator on minimal submanifolds.

Using (\ref{g11}) and (\ref{g12}) we have 
\[Lu(x_0)\geq g^{nn}u_{nn}=-g^{nn}\frac{|Du|(\eta_n|Du|+2Auu_n)}{nu_n}\]
which shows
\begin{equation}
    Lu(x_0)\geq -g^{nn}\frac{|Du|6n}{\eta}. \label{hh}
\end{equation}

Recalling the definition of $w$, we note the following for all  $1\leq i\leq n$ at $x_0$
\begin{align}
    w_{i}=\eta |Du|_i+\eta_i|Du|+2Auu_i \nonumber\\
    Lw=|Du|L\eta+2\sum_{a=1}^{n}g^{aa}\eta_{a}|Du|_{a}+\eta L(|Du|)+2AuLu+2A\sum_{a=1}^{n}g^{aa}u_a^2. \label{g5}
\end{align}

Next we observe the following
\begin{align}
\sum_{a,b=1}^{n}g^{ab}\partial_{ab}|Du|\nonumber\\
=\sum_{a,b,i=1}^{n}g^{ab}\left[  \frac{u_i u_{abi}}{|Du|}+\frac{u_{ai}u_{bi}}{|Du|}-\sum_{j=1}^{n}\frac{u_{aj}u_{bj}u_iu_j}{|Du|^3}\right]\nonumber\\
=\sum_{i=1}^{n}\frac{\psi_iu_i}{|Du|}+\sum_{a,b,i=1}^{n}g^{ab}\left[\frac{u_{ai}u_{bi}}{|Du|}-\sum_{j=1}^{n}\frac{u_{aj}u_{bj}u_iu_j}{|Du|^3}\right] \nonumber
\end{align}
where we get the last inequality using (\ref{linearize}).
Assuming $D^2u$ is diagonalized at $x_0$, we get 
\begin{equation}
    L(|Du|)(x_{0})= \sum_{i=1}^{n}\frac{\psi_iu_i}{|Du|}+\sum_{a=1}^{n}g^{aa}\frac{(|Du|^2-u_a^2)\lambda_a^2}{|Du|^3}\geq \sum_{i=1}^{n}\frac{\psi_iu_i}{|Du|}. \label{2.7}
\end{equation}
 We plug (\ref{2.7}) in (\ref{g5}) and on applying (\ref{hh}), (\ref{g11}), (\ref{g12}) we get the following at $x_0$
\begin{align*}
    Lw\geq -2ng^{nn}|Du|-2\sum_{a=1}^n g^{aa}\eta_a(\frac{\eta_a|Du|+2Auu_a}{\eta})+\\
    +\eta\sum_{i=1}^{n}\frac{\psi_iu_i}{|Du|}-6ng^{nn}\frac{|Du|}{\eta} 
    +\frac{2A}{n}g^{nn}|Du|^{2}\\
    \geq -2ng^{nn}|Du|-8g^{nn}\frac{|Du|}{\eta}-8g^{nn}Au\frac{|Du|}{\eta}+\eta \sum_{i=1}^{n}\frac{\psi_iu_i}{|Du|}\\
    -6ng^{nn}\frac{|Du|}{\eta}
    +\frac{2A}{n}g^{nn}|Du|^2.
\end{align*}
Noting that $Lw(x_0)\leq 0$, we divide the above inequality by $g^{nn}\frac{|Du|}{\eta}$ and on using (\ref{q19}), (\ref{q29}) we get the following at $x_0$
\begin{align*}
    0\geq -2n\eta-8-8Au-6n+\frac{2A}{n}\eta|Du|+\sum_{i=1}^{n}\frac{\psi_iu_i}{|Du|}\frac{\eta^2}{g^{nn}|Du|}\\
    \implies \eta|Du|\leq (12n+4)M-\sum_{i=1}^{n}\frac{\psi_iu_i}{2|Du|^{2}g^{nn}}\eta^{2}M\\
    \implies \eta|Du(x_0)|\leq (12n+4)M+\frac{|D\psi|}{2\eta|Du(x_0)|}\eta^3MC(\delta)
\end{align*}
where the last inequality follows from the Cauchy-Schwarz inequality and (\ref{k}).
Now solving the above quadratic expression in $\eta|Du(x_0)|$, we get
\[\eta|Du|(x_0)\leq C_1M+\sqrt{C_1^2M+C_2M}
\]
where \begin{align*}
    C_1=12n+4=C(n)\\
    C_2=C(n,||\psi||_{C^{1}(B_1)},\delta).
\end{align*}
This proves the gradient estimate: 
\begin{align*}
    |Du(0)|\leq w(0)\leq w(x_0)\leq \eta|Du|(x_0)+Au^2(x_0)\\
    \leq C_1M+\sqrt{C_1^2M^2+C_2M}+4nM\\
    \leq C(n,||\psi||_{C^{1}(B_1)},\delta)osc_{B_1}u+C(n).
\end{align*}
\end{proof}
\begin{remark} The above proof also follows from the observation that when $\psi$ lies in the supercritical range, $u$ is semiconvex [Lemma \ref{y1}]. On modifying $u$ to the convex function $
\tilde{u}(x)=u(x)+\cot(\delta)\frac{|x|^2}{2}$, the gradient estimate (which is independent of the Lipschitz norm of $\psi$) follows from the fact that the gradient of any convex
function is dominated by its oscillation.
\end{remark}
\section{The Jacobi inequality}
In this section we prove the Jacobi inequality and the integral Jacobi inequality, which is essential in proving the Hessian estimates. 
\begin{proposition}
Let $u$ be a smooth solution of (\ref{s}) in $\mathbb{R}^{n}$. Suppose that the Hessian $D^{2}u$ is diagonalized and the eigenvalue $\lambda_\gamma$ is distinct from all other eigenvalues of $D^2u$ at point $x_0$. Then we have the following at $x_0$
\begin{equation}
    |\nabla_g \ln\sqrt{1+\lambda_\gamma^2}|^2=\sum_{k=1}^{n}\lambda_\gamma^2h_{rrk}^2\label{grad}
\end{equation} and
\begin{align}
    \Delta_g\ln\sqrt{1+\lambda_\gamma^2}=\nonumber\\
    (1+\lambda_\gamma^2)h_{rrr}^2+\sum_{k\neq r}\left[\frac{2\lambda_\gamma}{\lambda_\gamma-\lambda_k}+\frac{2\lambda_\gamma^2\lambda_k}{\lambda_\gamma-\lambda_k}         \right]h_{kkr}^2 \nonumber\\
    +\sum_{k\neq r}\left[1+\frac{2\lambda_\gamma}{\lambda_\gamma-\lambda_k}+\frac{\lambda_\gamma^2(\lambda_k+\lambda_\gamma)}{\lambda_\gamma-\lambda_k}         \right]h_{rrk}^2\nonumber\\
    +\sum_{k>j,k,j\neq r}2\lambda_\gamma\left[  \frac{1+\lambda_k^2}{\lambda_\gamma-\lambda_k}+\frac{1+\lambda_j^2}{\lambda_\gamma-\lambda_j} +(\lambda_j+\lambda_k) \right]h^2_{kjr}\nonumber\\
    +\frac{\lambda_\gamma}{1+\lambda_\gamma^2}\psi_{\gamma\gamma}- \sum_{a=1}^n\lambda_ag^{aa}\psi_a\partial_{a}\ln\sqrt{1+\lambda_\gamma^2}. \label{cc}
\end{align}
\end{proposition} 
\begin{proof}
Define 
\[b_\gamma=\ln\sqrt{1+\lambda_\gamma^2}.
\] We assume $\gamma=1$ for the sake of simplifying notation. On implicitly differentiating the characteristic equation 
\[\det(D^2u-\lambda_1I)=0
\] near any point where $\lambda_1$ is distinct from the other eigenvalues we get 
\begin{align*}
    \partial_e \lambda_1=\partial_e u_{11}\\
    \partial_{ee}\lambda_1=\partial_{ee}u_{11}+\sum_{k>1}2\frac{(\partial_eu_{ik})^2}{\lambda_1-\lambda_k}
\end{align*} where $e$ is any arbitrary unit vector in $\mathbb{R}^n$.
On computing the derivatives of the smooth function $b_1$ near $x_0$, we get
\[|\nabla_g b_1|^2=\sum_{k=1}^n g^{kk}\left[\frac{\lambda_1}{1+\lambda^2}\partial_ku_{11} \right]^2=\sum_{k=1}^n\lambda_1^2h_{11k}^2.
\]
We see that 
\begin{align*}
    \partial_{ee}b_1=\partial_{ee}\ln\sqrt{1+\lambda_1^2}=\frac{\lambda_1}{1+\lambda_1^2}\partial_{ee}\lambda_1+\frac{1-\lambda_1^2}{(1+\lambda_1^2)^2}(\partial_e\lambda_1)^2.\\
   \text{ At $x_0$, } \partial_{ee}b_1=\frac{\lambda_1}{1+\lambda_1^2}[\partial_{ee}u_{11}+\sum_{k>1}2\frac{(\partial_eu_{ik})^2}{\lambda_1-\lambda_k}]+\frac{1-\lambda_1^2}{(1+\lambda_1^2)^2}(\partial_eu_{11})^2 .
\end{align*}
We define an operator $L=\sum_{a,b=1}^{n}g^{ab}\partial_{ab}.$ At $x_0$, we have
\begin{align}
    Lb_1=\sum_{r=1}^{n}g^{rr}\partial_{rr}b_1\label{label}\\
    =\sum_{r=1}^{n}g^{rr}\frac{\lambda_1}{1+\lambda_1^2}\left[\partial_{rr}u_{11}+2\sum_{k>1}\frac{u_{1kr}^2}{\lambda_1-\lambda_k}    \right]\nonumber\\
    +\sum_{r=1}^n\frac{1-\lambda_1^2}{(1+\lambda_1^2)^2}g^{rr}u_{11r}^2. \label{bbb22}
    \end{align}
    Combining (\ref{linearize}) with (\ref{111}) we observe the following at $x_0$ for $i,j$ fixed 
\begin{align}
    Lu_{ij}=\sum_{a,b=1}^{n}g^{ab}u_{ijab}=\psi_{ij}-g^{ab}_i u_{abj}\nonumber\\
    =\psi_{ij}+\sum g^{aa}g^{bb}(\lambda_a+\lambda_b)u_{abi}u_{abj}. \label{u}
\end{align}

Next in (\ref{u}) we substitute $\partial_{rr} u_{11}$ in terms of lower order derivatives and $\psi$. Recalling (\ref{2!}), at $x_0$ we have 
\begin{equation}
    \Delta_g = \sum_{i=1}^n g^{ii}\partial_{ii} -
     \sum_{i=1}^ng^{ii}\lambda_i\psi_i\partial_{i}. \label{ho}
\end{equation}
We set $i=j=1$, and plug (\ref{u}) in (\ref{bbb22}) and then regroup the terms $h_{**1}, h_{11*}, h_{*@1}$ to get 
\begin{align}
    \Delta_g b_1=Lb_1- \sum_{a=1}^n\lambda_ag^{aa}\psi_a\partial_{a}\ln\sqrt{1+\lambda_1^2}=\frac{\lambda_1}{1+\lambda_1^2}\psi_{11}+\\
    (1-\lambda_1^2)h_{111}^2+2\sum_{a=1}^{n}\lambda_1\lambda_ah^2_{aa1}+2\sum_{k>1}\frac{\lambda_1(1+\lambda_k^2)}{\lambda_1-\lambda_k}h_{kk1}^2\\
    +2\sum_{k>1}\lambda_1(\lambda_1+\lambda_k)h_{k11}^2+\sum_{k>1}(1-\lambda_1^2)h_{11k}+
    2\sum_{k>1}\frac{\lambda_1(1+\lambda_k^2)}{\lambda_1-\lambda_k}h_{1k1}^2\\
    +2\sum_{k>j>1}\lambda_1(\lambda_j+\lambda_k)h_{jk1}^2+ 2\sum_{j\neq k,j,k>1}\frac{\lambda_1(1+\lambda_k^2)}{\lambda_1-\lambda_k}h_{jk1}^2- \sum_{a=1}^n\lambda_ag^{aa}\psi_a\partial_{a}\ln\sqrt{1+\lambda_1^2}.
\end{align}
On simplifying we get (\ref{cc}).
\end{proof}

\begin{lemma}\label{4.2}
Let $u$ be a smooth solution of (\ref{s}) in $\mathbb{R}^{n}$. Suppose that the Hessian $D^{2}u$ is diagonalized at $x_0$ and that the ordered eigenvalues $\lambda_1\geq\lambda_2\geq...\geq\lambda_n$ of the Hessian satisfy $\lambda_1=...=\lambda_m>\lambda_{m+1}$ at $x_0$. Then the function $b_m=\frac{1}{m}\sum_{i=1}^m\ln\sqrt{1+\lambda_i^2}$ is smooth near $x_0$ and at $x_0$ it satisfies 
\begin{equation}\Delta_g b_m\geq c(n)|\nabla_gb_m|^2-C\label{J}
\end{equation}
where $C=C(||\psi||_{C^{1,1}(B_1)},\delta,n)$.

\end{lemma}
\begin{proof}
\begin{itemize}
    \item [Step 1.] The function $b_m$ is symmetric in $\lambda_1,..,\lambda_m$. Thus for
$m < n$, $b_m$ is smooth in terms of the matrix entries when $\lambda_m>\lambda_{m+1}$ and it is 
still smooth in terms of $x$ after being composed with the smooth function $D^2u(x)$, in particular near $x_0$, at which $\lambda_1=...=\lambda_m>\lambda_{m+1}$. For $m= n$, $b_n$ is
clearly smooth everywhere.
First let's assume that the first $m$ eigenvalues are distinct. Using (\ref{bbb22}) from
the Proposition above, we compute $mL(b_m)$ where $L$ is the operator defined in (\ref{label}). As before after grouping the terms $h_{***},h_{**@},h_{*@!}$ in the summation, we get 
\begin{align*}
     mL(b_m)(x_0)=\sum_{k\leq m}(1+\lambda_k^2)h^2_{kkk}+(\sum_{i<k\leq m}+\sum_{k<i\leq m})(3+\lambda_i^2+2\lambda_i\lambda_k)h^2_{iik}\\
     +\sum_{k\leq m<i}\frac{2\lambda_k(1+\lambda_k\lambda_i)}{\lambda_k-\lambda_i}h^2_{iik}
     +\sum_{i\geq m<k}\frac{3\lambda_i+\lambda_k+\lambda_i^2(\lambda_i+\lambda_k)}{\lambda_i-\lambda_k}h^2_{iik}\\
     +2\Bigg[\sum_{i<j<k\leq m}(3+\lambda_i\lambda_j+\lambda_j\lambda_k+\lambda_k\lambda_i)h^2_{ijk}\\
     +\sum_{i<j\leq m<k}(1+\lambda_i\lambda_j+\lambda_j\lambda_k+\lambda_k\lambda_i+\lambda_i\frac{1+\lambda_k^2}{\lambda_i-\lambda_k}+\lambda_j\frac{1+\lambda^2_k}{\lambda_j-\lambda_k})h^2_{ijk}\\
     +\sum_{i\leq m<j<k}\lambda_i[\lambda_j+\lambda_k+\frac{1+\lambda_j^2}{\lambda_i-\lambda_j}+\frac{1+\lambda_k^2}{\lambda_j-\lambda_k}]h_{ijk}^2\Bigg]\\
     + \sum_{i=1}^m\frac{\lambda_i}{1+\lambda_i^2}\psi_{ii}.
     \end{align*}
     
     Observe that as a function of matrices, $b_m$ is $C^2$ at $D^2u(x_0)$ with eigenvalues satisfying $\lambda=\lambda_1=...=\lambda_m>\lambda_{m+1}$. Note that $D^2u(x_0)$ can be approximated by matrices with distinct eigenvalues. This shows that the above expression for $Lb_m(x_0)$ still holds good. Using Lemma \ref{y1} we can further simplify it to
\begin{align}
     mL(b_m)(x_0)= \sum_{k\leq m}(1+\lambda^2)h^2_{kkk}+(\sum_{i<k\leq m}+\sum_{k<i\leq m})(3+3\lambda^2)h_{iik}^2+\sum_{k\leq m<i}\frac{2\lambda(1+\lambda\lambda_i)}{\lambda-\lambda_i}h^2_{iik}\nonumber\\
     +\sum_{i\leq m<k}\frac{3\lambda-\lambda_k+\lambda^2(\lambda+\lambda_k)}{\lambda-\lambda_k}h_{iik}^2
     +2\Bigg[\sum_{i<j<k\leq m}(3+3\lambda^2)h_{ijk}^2
     +\sum_{i<j\leq m<k}[1+\frac{2\lambda}{\lambda-\lambda_k}+\nonumber\\
     \frac{\lambda^2(\lambda+\lambda_k)}{\lambda-\lambda_k}]h_{ijk}^2
     +\sum_{i\leq m<j<k}\lambda[\lambda_j+\lambda_k+\frac{1+\lambda_j^2}{\lambda-\lambda_j}+\frac{1+\lambda_k^2}{\lambda-\lambda_k}]h_{ijk}^2\Bigg]
     +\sum_{i=1}^{m}\frac{\lambda}{1+\lambda^2}\psi_{ii}\nonumber\\
     \geq \sum_{k\leq m}\lambda^2h_{kkk}^2+(\sum_{i<k\leq m}+\sum_{k<i\leq m})3\lambda^2h^2_{iik}+\sum_{k\leq m<i}\frac{2\lambda^2\lambda_i}{\lambda-\lambda_i}h^2_{iik}\nonumber\\
     +\sum_{i\leq m<k}\frac{\lambda^2(\lambda+\lambda_k)}{\lambda-\lambda_k}h^2_{iik}+\sum_{i=1}^m\frac{\lambda_i}{1+\lambda_i^2}\psi_{ii}.\label{star2}
 \end{align}

 Using the $C^1$ continuity of $b_m$ as a function of matrices at $D^2u(x_0)$, we can simplify (\ref{grad}) at $x_0$ to 
 \begin{equation}
     |\nabla_gb_m|^2(x_0)=\frac{1}{m^2}\sum_{1\leq k\leq n}\lambda^2\left[ \sum_{i\leq m}h_{iik} \right]^2\leq \frac{\lambda^{2}}{m}\sum_{1\leq k\leq n}\left[ \sum_{i\leq m}h^2_{iik}\right]. \label{star1}
 \end{equation}

Combining (\ref{star1}) and (\ref{star2}) we get the following at $x_0$:
\begin{align}
    m(\Delta_g b_m-\varepsilon(n)|\nabla_gb_m|^2)\geq \nonumber\\
    \lambda^2\left[\sum_{k\leq m}(1-\varepsilon)h_{kkk}^2+(\sum_{i<k\leq m} + \sum_{k<i\leq m})(3-\varepsilon)h^2_{iik}+ 2\sum_{k\leq m< i}\frac{\lambda_i}{\lambda-\lambda_i}h^2_{iik} \right] \label{y5.23}\\
    +\lambda^2\left[  \sum_{i\leq m \leq k} (\frac{\lambda+\lambda_k}{\lambda-\lambda_k}-\varepsilon)h^2_{iik}   \right] \label{y5.24}\\
    +\sum_{i=1}^m\frac{\lambda_i}{1+\lambda_i^2}\psi_{ii}\label{11}\\
    -m\sum_{i=1}^n\lambda_ig^{ii}\psi_i\partial_i b_m \label{22}
    \end{align}
with $\varepsilon(n)$ to be fixed.
\item[Step 2.] Next we estimate each of the terms of the above expression. For each fixed $k$ in the above expression, we set 
$t_i=h_{iik}.$
For the sake of simplicity, we use the following notation 
 \begin{align*}
 (\ref{y5.23}+\ref{y5.24})=Z_1+Z_2=Z\\
    H^k(x_0)=t_1(x_0)+...+t_{n-1}(x_0)+t_n(x_0)=t'(x_0)+t_{n}(x_0)
     \end{align*}
     where $H^k$ denotes the $k$th component of the mean curvature vector given by (\ref{mean}), i.e. the component of the mean curvature vector along $J(e_k,Du_{e_k})$ with $e_k$ being the $k^{th}$ eigendirection of $D^2u$.
    So far we have 
 \[m(\Delta_g b_m-\varepsilon(n)|\nabla_gb_m|^2)\geq 
    Z+\sum_{i=1}^m\frac{\lambda_i}{1+\lambda_i^2}\psi_{ii}-m\sum_{i=1}^n\lambda_ig^{ii}\psi_i\partial_i b_m.\]

 \item[Step 2.1. ]We estimate the term $Z$ by first showing that $Z_1\geq-C(||\psi||_{C^{1}(B_1)})$.
 For each fixed $k\leq m$ in (\ref{y5.23}), we show that the $[\hspace{.1cm}]_k$ term is $\geq 0$. 
 For the case where $\lambda_i\geq 0$ for all $i$, the proof follows directly. So we consider only the case where $\lambda_{n-1}>0>\lambda_n$. 
 For simplifying notation we assume $k=1$. Noting that  $t_n(x_0)=H^1(x_0)-t'(x_0)$ we observe the following:
 \begin{align}
     [\hspace{.1cm}]_1= \Bigg[ (1-\varepsilon)t_1^2+\sum_{i=2}^m(3-\varepsilon)t_i^2+\sum_{i=m+1}^{n-1}\frac{2\lambda_i}{\lambda-\lambda_i}t_i^2 \Bigg]+\frac{2\lambda_n}{\lambda-\lambda_n}t_n^2\nonumber\\
     = \Bigg[ (1-\varepsilon)t_1^2+\sum_{i=2}^m(3-\varepsilon)t_i^2+\sum_{i=m+1}^{n-1}\frac{2\lambda_i}{\lambda-\lambda_i}t_i^2   \Bigg]\nonumber\\
     +\frac{2\lambda_n}{\lambda-\lambda_n}[(H^1)^2-2H^1t'+t'^2]\nonumber\\
     \geq \Bigg[ (1-\varepsilon)t_1^2+\sum_{i=2}^m(3-\varepsilon)t_i^2+\sum_{i=m+1}^{n-1}\frac{2\lambda_i}{\lambda-\lambda_i}t_i^2   \Bigg]\nonumber\\
     +\frac{2\lambda_n}{\lambda-\lambda_n}[t'^2(1+\delta)]
     +\frac{2\lambda_n}{\lambda-\lambda_n}[(H^1)^2(1+\frac{4}{\delta})]\nonumber
     \end{align} where the last inequality follows from Young's inequality. Noting that $\frac{2\lambda_n}{\lambda-\lambda_n}\geq-\frac{2}{n}$ and (\ref{mean}), we see the following
     \begin{align}
     [\hspace{.1cm}]_1
      \geq \Bigg[ (1-\varepsilon)t_1^2+\sum_{i=2}^m(3-\varepsilon)t_i^2+\sum_{i=m+1}^{n-1}\frac{2\lambda_i}{\lambda-\lambda_i}t_i^2   \Bigg]\nonumber\\
     +\frac{2\lambda_n}{\lambda-\lambda_n}[t'^2(1+\delta)]-C(||\psi||_{C^{1}(B_1)})\nonumber\\
     \geq \Bigg[(1-\varepsilon)t_1^2+\sum_{i=2}^m(3-\varepsilon)t_i^2+\sum_{i=m+1}^{n-1}\frac{2\lambda_i}{\lambda-\lambda_i}t_i^2  \Bigg]\label{q1}\\
     \Bigg[1+\frac{2(1+\delta)\lambda_n}{\lambda-\lambda_n}\big(\frac{1}{1-\varepsilon}+\sum_{i=2}^m\frac{1}{3-\varepsilon}+\sum_{i=m+1}^{n-1}\frac{\lambda-\lambda_i}{2\lambda_i} \big) \Bigg]\label{q2}\\
     -C(||\psi||_{C^{1}(B_1)})\nonumber
 \end{align}where the last inequality follows from the Cauchy-Schwartz inequality. We see that (\ref{q1}) is positive, so now we need to choose $\varepsilon(n)$ suitably to make (\ref{q2}) positive, thereby proving $Z_1\geq -C(||\psi||_{C^1(B_1)})$.
 We have
 \begin{align}
     \Bigg[1+\frac{2(1+\delta)\lambda_n}{\lambda-\lambda_n}(\frac{1}{1-\varepsilon}+\sum_{i=2}^m\frac{1}{3-\varepsilon}+\sum_{i=m+1}^{n-1}\frac{\lambda-\lambda_i}{2\lambda_i}) \Bigg]\nonumber\\
     =\frac{2(1+\delta)\lambda_n}{\lambda-\lambda_n}\Bigg[ \frac{\lambda-\lambda_n}{2\lambda_n}-\frac{\lambda-\lambda_n}{(\frac{2}{\delta}+2)\lambda_n} +\frac{1}{1-\varepsilon}+\frac{m-1}{3-\varepsilon}
     +\frac{\lambda-\lambda_{m+1}}{2\lambda_{m+1}}+...+\frac{\lambda-\lambda_{n-1}}{2\lambda_{n-1}}\Bigg]\nonumber\\
     =\frac{2(1+\delta)\lambda_n}{\lambda-\lambda_n}\Bigg[ \frac{1}{1-\varepsilon}+\frac{m-1}{3-\varepsilon}+\frac{\lambda}{2}(\frac{1}{\lambda_1}+..+\frac{1}{\lambda_1})-\frac{n}{2} \Bigg]-\delta\nonumber\\
     =\frac{2(1+\delta)\lambda_n}{\lambda-\lambda_n}\Bigg[ \frac{1}{1-\varepsilon}+\frac{m-1}{3-\varepsilon}+\frac{\lambda}{2}\frac{\sigma_{n-1}}{\sigma_n}-\frac{n}{2}
     \Bigg]-\delta\nonumber\\
     \geq\frac{2(1+\delta)\lambda_n}{\lambda-\lambda_n}\Bigg[ \frac{1}{1-\varepsilon}+\frac{m-1}{3-\varepsilon}-\frac{n}{2}\Bigg]
     -\delta \nonumber
     \end{align}  where we used that the fact $\lambda_1=..=\lambda_m$ and Lemma \ref{y1}. Now as $\delta$ is arbitrarily small,
     and $\lambda_n<0$ we choose $\varepsilon(n)>0$ such that 
     \begin{equation*}
         \Bigg[ \frac{1}{1-\varepsilon}+\frac{m-1}{3-\varepsilon}-\frac{n}{2}\Bigg]\leq 0
     \end{equation*} which in turn makes (\ref{q2}) positive. 
     On simplifying, we see that 
     \begin{equation*}
     \varepsilon(n)\leq 2-\frac{m}{n}-\sqrt{(1-\frac{m}{n})^2+\frac{4}{n}}.
     \end{equation*}

 \item[Step 2.2] Now we estimate the term $Z_2$. For each $k$ between $m$ and $n$, we have $\lambda_k>0$, and the $[\hspace{.1cm}]_k$ in (\ref{y5.24}) satisfies 
 \begin{align*}
     [\hspace{.1cm}]_k=\sum_{i\leq m}\Bigg[\frac{\lambda+\lambda_k}{\lambda-\lambda_k}-\varepsilon \Bigg]t_i^2\\
     \geq \sum_{i\leq m}(1-\varepsilon)t_i^2\geq 0
 \end{align*}
 assuming $\varepsilon\leq 1$.\\
 For $k=n$, the $[\hspace{.1cm}]_n$ term in (\ref{y5.24}) becomes \begin{align*}
     [\hspace{.1cm}]_n=\sum_{i\leq m}\Bigg[\frac{\lambda+\lambda_k}{\lambda-\lambda_k}-\varepsilon \Bigg]t_i^2\\
     \geq \sum_{i\leq m}\Bigg[ \frac{n-2}{n}-\varepsilon\Bigg]t_i^2\geq 0
 \end{align*}
 where the last inequality follows from Lemma \ref{y1} and the assumption of $\varepsilon\leq \frac{n-2}{n}.$
 So far we have shown $Z\geq-C(||\psi||_{C^{1}(B_1)})$ for $n-1\geq m\geq 1$. When $m=n$, we have $\lambda_1=...=\lambda_n>0$ and therefore, $Z\geq-C(||\psi||_{C^{1}(B_1)})$ holds.

 \item[Step 2.3] Next we estimate the remaining terms (\ref{11}) and (\ref{22})
\begin{align*}
    m(\Delta_g b_m-\varepsilon(n)|\nabla_gb_m|^2)\geq \\
    Z+\sum_{i=1}^m\frac{\lambda_i}{1+\lambda_i^2}\psi_{ii}-m\sum_{i=1}^n\lambda_ig^{ii}\psi_i\partial_i b_m\geq\\
    \geq-C(||\psi||_{C^{1}(B_1)})+\sum_{i=1}^m\frac{\lambda_i}{1+\lambda_i^2}\psi_{ii}-\frac{\delta}{2}m^2|\nabla_g b_m|^2-\sum_{i=1}^m\frac{2}{\delta}g^{ii}\lambda^2_{i}\psi_i^2
    \end{align*}
    where the last inequality follows from Young's inequality and $Z\geq-C(||\psi||_{C^{1}(B_1)})$.
Denoting $c(n)=\varepsilon(n)+\delta/2$, we get 
\begin{align}
    \Delta_g b_m-c(n)|\nabla_gb_m|^2\nonumber\\ \geq-C(||\psi||_{C^{1}(B_1)})-|\sum_{i=1}^m\frac{\lambda_i}{1+\lambda_i^2}[\psi_{ii}-\frac{2}{\delta}\lambda_i\psi_i^2]|\nonumber\\
    \geq -C(||\psi||_{C^{1,1}(B_1)},\delta,n)=-C. \label{C}
\end{align}

\end{itemize}
\end{proof}

\subsection{The integral Jacobi inequality}
In order to prove Hessian estimates we will need the following integral form of the Jacobi inequality (\ref{J}).
\begin{proposition}
Let $u$ be a smooth solution of (\ref{s}) on $B_R(0)\subset\mathbb{R}^{n}$ with 
$\psi\geq (n-2)\frac{\pi}{2}+\delta$. Let
\begin{equation}
b=\ln\sqrt{1+\lambda_{\max}^2} \label{star}
\end{equation} where 
 $\lambda_{\max}$ is the largest eigenvalue of $D^2u$, namely, $\lambda_{\max}=\lambda_1\geq\lambda_2\geq...\geq\lambda_n$. Then for all non-negative $\phi\in C^{\infty}_0(B_R)$, $b$ satisfies the integral Jacobi inequality 
\begin{equation}
    \int_{B_R}-\langle \nabla_g \phi,\nabla_g b\rangle_g dv_g\geq c(n)\int_{B_R}\phi|\nabla_gb|^2dv_g-\int_{B_R}C\phi dv_g \label{IJ}
\end{equation}
where $C$ and $c(n)$ are the constants from (\ref{J}).
\end{proposition}
\begin{proof}
If $b_1=\ln\sqrt{1+\lambda_{\max}^2} $ is smooth everywhere, then the pointwise Jacobi inequality (\ref{J}) implies the integral Jacobi inequality (\ref{IJ}). It is
easy to see that $\lambda_{max}$ is always a Lipschitz function of the entries of the Hessian of $u$.
Since $u$ is smooth in $x$, $b$ is Lipschitz in terms of $x$. We now show that (\ref{J}) holds in the viscosity sense. \\
Let $x_0\in B_R(0)$ and let $P(x)$ be a quadratic polynomial such that
\[P(x)\geq b(x) \text{ with equality holding at $x_0$}.
\]
Now if $x_0$ is a smooth point of $b$, then from (\ref{J}), we see that at $x_0$, with $m=1$, the following holds
\[\Delta_g P\geq c(n)|\nabla_gP|^2-C(||\psi||_{C^{1,1}(B_1)},\delta,n).
\]
Or else we would have $m>1$, i.e. $\lambda_1$ is not distinct at $x_0$. Let's suppose that we have 
$\lambda_1=...=\lambda_k>\lambda_{k+1}$ at $x_0$. Consider the function $b_k=\frac{1}{k}\sum_{i=1}^k \ln \sqrt{1+\lambda_i^2}$. Note that this is smooth near $x_0$ from Lemma \ref{4.2}. Observe that since 
$b(x)\geq b_k(x)$ with equality holding at $x_0$, we must have 
\[P(x)\geq b_k(x) \text{ with equality holding at $x_0$}.
\]
So now applying the pointwise Jacobi inequality (\ref{J}) to  $b_k$, we see the following holds at $x_0$
\[\Delta_g P\geq c(n)|\nabla_gP|^2-C(||\psi||_{C^{1,1}(B_1)},\delta,n).
\]  So far we have shown that (\ref{J}) holds in the viscosity sense. Applying Ishii's result \cite[Theorem 1]{Ish} we see that the viscosity subsolution $b$ of (\ref{J}) should also be a distribution subsolution. Now since $b$ is a Lipschitz function we perform integration by parts on the first term to get
\[\int_{B_R}\Delta_g b \phi dv_g=\int_{B_R}-\langle \nabla_g \phi,\nabla_g b\rangle_g dv_g.
\]
On combining \cite[Theorem 1]{Ish} and the above equation, we get $(\ref{IJ})$.

 \end{proof}

\section{Hessian Estimates}

\subsection{Mean Value Inequality}
In \cite[Theorem 3.4]{MS}, Michael-Simon established the mean value inequality for a non-negative subharmonic function on a $m$-dimensional submanifold $M$ of $\mathbb{R}^n$. In order to prove the Hessian estimates, having the following mean value inequality for a variable Lagrangian phase is crucial. For completeness, we include a proof here.
\begin{proposition}
Let $M\subset\mathbb{R}^{n}$ be a $m$-dimensional submanifold, $0\in M$, and $s>0$ satisfies $B_s(0)\cap\partial M=\emptyset$. Suppose that there exists $\Lambda_0>0$ such that $|\vec{H}|\leq \Lambda_0$ where $\vec{H}$ is the mean curvature vector of $M$. If $f$ is a non-negative function on $M$ such that $\Delta_{g}f\geq-\beta f$, then  
\begin{equation}
    f(0)\leq C(\beta,n,\Lambda_0)\frac{\int_{B_1\cap M}f dv_g}{Vol(B_1\subset\mathbb{R}^{n})}. \label{mvi}
\end{equation}
\end{proposition}

\begin{proof}
The symmetric matrix $\tilde{g}^{ij}(x)$ denotes the projection of $\mathbb{R}^{n}$ onto the $m$-dimensional submanifold $M$, and it satisfies 
\begin{align}
    \sum_{i=1}^{n}\tilde{g}^{ii}(x)=m \label{gec}\\
    0\leq\sum_{i,j=1}^{n}\tilde{g}^{ij}(x)x_ix_j\leq|x|^2 \hspace{.2cm}\forall x\in \mathbb{R}^{n} \nonumber.
\end{align} Let $U$ denote an open subset of $\mathbb{R}^{n}$ that contains $M$.
Let $\phi$ be a non decreasing $C^1(\mathbb{R})$ function such that $\phi(t)=0$ when $t\leq 0$. For each $x_0\in M$, we define the following two functions 
\begin{align*}
    g_0(y)=\int_M f(x)\phi(y-r)dv_g(x)\\
    h_0(y)=\int_Mf(x)|H(x)|\phi(y-r)dv_g(x)
\end{align*}
where $r=|x-x_0|$. Let's assume $d(x_0,\partial U)=1.$
\begin{claim}
For $0<y\leq 1$, we prove that
\begin{equation}
    -\frac{d}{dy}\left[ \frac{\ g_0(y)}{y^m}\right]\leq y^{-m-1}\int_{0}^{y}t  h'_0(t)dt. \label{main claim}
\end{equation}
\end{claim}
Let $\gamma$ be a real valued function defined by
\[\gamma(s)=\int_s^{\infty}t\phi(y-t)dt.
\]
Then $\gamma(s)=0$ when $s\geq y$. Note that $\gamma(|x-x_0|)$ is in $C^2(U)$ and vanishes outside a compact subset of $U$ if $y<1$. So for $y<1$ we can use $\gamma(r)$ as a test function. Next we observe that 
\begin{align*}
\Delta_g \gamma(r)=-\left[\phi(y-r)\sum_{i=1}^{n}\tilde{g}^{ii}-r\phi'(y-r)\sum_{i,j=1}^{n}\tilde{g}^{ij}(\frac{x_i-(x_0)_i}{r})(\frac{x_j-(x_0)_j}{r})         \right]-\\
\phi(y-r)\sum_{i=1}^{n}(x_i-(x_0)_i)H_i
\end{align*} where $H_i$'s denote the components of the mean curvature vector. So then by our assumption $\Delta_{g}f\geq-\beta f$ and (\ref{gec}) we have 
\begin{align*}
mg_0(y)-\int_M fr\phi'(y-r)dv_g(x)\leq \int_M f|H|r\phi(y-r)dv_g+\int_M \beta f\gamma(r)dv_g\\
\implies mg_0(y) -\int_M fr\phi'(y-r)dv_g(x)\leq \int_M f|H|r\phi(y-r)dv_g+\int_M \beta f y\phi(y-r) dv_g.
\end{align*}
In the last inequality we used $\gamma(r)\leq \int _{r}^1 t\phi(y-t)dt\leq y\phi(y-r)$ since we need $t\leq y$ for the function to be non-zero.
This gives us
\begin{equation}
mg_0(y)-\int_M fr\phi'(y-r)dv_g(x)\leq \int_M f|H|r\phi(y-r)dv_g+\int_M \beta f y\phi(y-r) dv_g. \label{imp1}
\end{equation}
Using the inequality $r\phi'(y-r)\leq y\phi'(y-r)$ we get
\[\int_M fr\phi'(y-r)dv_g\leq yg_0'(y).
\]
We see that 
\begin{align*}
   \int_M f|H|r\phi(y-r)dv_g =\int_M f|H|\left[ \int_0^y r\phi'(t-r)dt \right]dv_g \\
    \leq \int_M f|H|\left[ \int_0^y t\phi'(t-r)dt \right]dv_g\\
    =\int_0^y th_0'(t)dt.
\end{align*}

Therefore, (\ref{imp1}) reduces to 
\[mg_0(y)-\int_M fr\phi'(y-r)dv_g(x)\leq \int_0^y th_0'(t)dt+\int_M \beta f y\phi(y-r) dv_g
\] which can be written as
\[-\frac{d}{dy}\left[ \frac{g_0(y)}{y^m}\right]\leq y^{-m-1}\int_0^yth_0'(t)dt+y^{-m}
\int_M \beta f \phi(y-r) dv_g.
\]
Using the fact
\[\int_0^yth_0'(t)dt\leq y\int_0^yh'_0(t)dt=yh_0(y)
\]
we get
\[-\frac{d}{dy}\left[ \frac{g_0(y)}{y^m}\right]\leq \frac{s_0(y)}{y^m}.
\]
Observe that
\begin{align*}
    s_0=\int_Mf(x)|H(x)|\phi(y-r)dv_g(x)+\int_M \beta f \phi(y-r) dv_g\\
    \leq C\int_M f\phi(y-r)dv_g(x)
\end{align*}
where $C=(\beta,n,\Lambda_0)$.
This shows
\[-\frac{d}{dy}\left[ \frac{g_0(y)}{y^m}\right]\leq C\left[ \frac{g_0(y)}{y^m}\right].
\]

Integrating this we get 
\[\sup_{t\in (0,y)}\left[ \frac{g_0(t)}{t^m}\right]\leq e^{Cy}\left[ \frac{g_0(y)}{y^m}\right]
\] for all $y\in (0,1)$. Now we choose $\phi$ such that $\phi(s)=1$ when $s\geq \varepsilon$ and letting $\varepsilon\rightarrow 0+$, (\ref{mvi}) follows.

\end{proof}

\begin{remark} The mean value inequality for $b$ defined in (\ref{star}) also follows from the observation that the potential $u$ is semi convex, so a rotation of Yuan \cite[Pg 125]{YY02} can be performed on 
the gradient graph $(x,Du(x))$, which results in a uniformly elliptic Laplace-Beltrami operator on the rotated graph $(\bar x,D\bar u(\bar x))$. 
Then on applying the De Giorgi iteration for divergence form equations \cite[Pg 197, (8.58)]{GT} to $b$, and given the invariance of the Jacobi inequality and integral, we obtain a MVI for it. 
\end{remark}

Note that the condition $u$ is smooth in section 4 can be clearly replaced by $u\in C^4$. Now we prove our main Theorem.

\subsection{Proof of Theorem \ref{main1}}
\begin{proof} 
We first consider the case $n\neq 2$. We verify that the positive function $b$ defined in (\ref{star}) satisfies the requirements of the above mean value inequality. Observe that by our choice of $b$, we have 
\begin{equation}
    b\geq \ln\sqrt{1+\tan^2(\frac{\pi}{2}-\frac{\pi}{n})}\geq \ln \sqrt{4/3} \label{b}.
\end{equation}
Combining the above with (\ref{J}), we conclude that $b$ satisfies the conditions of the above Theorem:
\begin{align*}
    \Delta_g b\geq c(n)|\nabla_gb|^2-C
    \geq  -\frac{C}{\ln \sqrt{4/3}}b=-Cb
\end{align*}
where $C=C(n,||\psi||_{C^{1,1}(B_1)},\delta, \ln \sqrt{4/3})$ is the positive constant from (\ref{IJ}). 

For simplifying notation in the remaining proof, we assume $R=2n+1$ and $u$ is a solution on $B_{2n+1}\subset\mathbb{R}^n$. Then by scaling $v(x)=\frac{u(\frac{R}{2n+1}x)}{(\frac{R}{2n+1})^2}$, we get the estimate in Theorem \ref{main1}. Also, we will denote the above constant $C$ as $C(n,\psi,\delta)$. 
\begin{itemize}
    \item[Step 1.] We show that the function $b^{\frac{n}{n-2}}$ meets the requirements of the above MVI:
    \begin{align*}
        \int-\langle\nabla_g \phi,\nabla_g b^{\frac{n}{n-2}}\rangle_g dv_g\\
        =\int-\langle \nabla_g(\frac{n}{n-2}b^\frac{2}{n-2}\phi)-\frac{2n}{(n-2)^2}b^{\frac{4-n}{n-2}}\phi\nabla_gb,\nabla_g b     \rangle_g dv_g\\
        \geq\int (\frac{n}{n-2}c(n)b^{\frac{2}{n-2}}\phi|\nabla_g b|^2+\frac{2n}{(n-2)^2}b^{\frac{4-n}{n-2}}\phi|\nabla_g b|^2)dv_g-\int C\frac{n}{n-2}\phi b^{\frac{2}{n-2}}dv_g \\
        \geq -\int C\frac{n}{n-2}\frac{1}{b}b^{\frac{n}{n-2}} \phi dv_g \\
        \geq -C(n,\psi,\delta)\int b^{\frac{n}{n-2}} \phi dv_g
    \end{align*}
    where the last inequality follows from (\ref{b}). 

So now by the MVI applied to the Lipschitz function $b^\frac{n}{n-2}$ we get
\begin{equation}b(0)\leq C(n,\psi,\delta)(\int_{\Tilde{B_1}\cap X}b^{\frac{n}{n-2}}dv_g)^{\frac{n-2}{n}}
\leq  C(n,\psi,\delta)(\int_{B_1}b^{\frac{n}{n-2}}dv_g)^{\frac{n-2}{n}}\label{bbb}
\end{equation} 
where $X=(x,Du(x))\subset \mathbb{R}^n\times\mathbb{R}^n$ is the Lagrangian submanifold, $\Tilde{B_1}$ is the ball with radius $1$ and center at $(0,Du(0))$ in $\mathbb R^n\times \mathbb R^n$, and $B_1$ is
the ball with radius $1$ and center at $0$ in $\mathbb R^n.$
Choose a cut off $\phi\in C^\infty_0(B_2)$ such that $\phi\geq 0$, $\phi=1$ on $B_1$ and $|D\phi|<2$. That gives us 
\begin{equation}
    [\int_{B_1}b^{\frac{n}{n-2}}dv_g]^{\frac{n-2}{n}}\leq [\int_{B_2}\phi^{\frac{2n}{n-2}} b^{\frac{n}{n-2}}dv_g]^{\frac{n-2}{n}}=[\int_{B_2}(\sqrt{b}\phi) ^{\frac{2n}{n-2}} dv_g]^{\frac{n-2}{n}} \label{b1}.
\end{equation} 

We assume $\sqrt{b}\phi$ to be $C^1$ by approximation. Applying the general Sobolev inequality \cite[Theorem 2.1]{MS} to it on this Lagrangian submanifold, and using the mean curvature formula (\ref{mean}), we get 
\begin{align}
[\int_{B_2}(\sqrt{b}\phi )^{\frac{2n}{n-2}}dv_g]^{\frac{n-2}{n}}\leq C(n)[\int _{B_2}|\nabla_g(\sqrt{b}\phi )|^2dv_g+\int_{B_2}|\sqrt{b}\phi \nabla_g \psi|^2dv_g]. \label{b2}
\end{align}
Next we observe the following
\begin{align}
    |\nabla_g(\sqrt{b}\phi )|^2=|\frac{1}{2\sqrt{b}}\phi\nabla_g b+\sqrt{b}\nabla_g\phi|^2\nonumber\\
    \leq \frac{1}{2b}\phi^2|\nabla_g b|^2+2b|\nabla_g\phi|^2\nonumber\\
    \leq \phi^2|\nabla_g b|^2+2b|\nabla_g \phi|^2.\label{b3}
    \end{align}
    Combining (\ref{b1}, \ref{b2}, \ref{b3}) and plugging  into (\ref{bbb}), we see that
    \begin{align}
    b(0)\leq C(n,\psi,\delta)[\int_{B_2}\phi^2|\nabla_g b|^2 dv_g+\int_{B_2}b|\nabla_g \phi|^2dv_g
    +\int_{B_2}\phi^2b|\nabla_g \psi|^2dv_g].\label{III}
\end{align}
\item[Step 2.] 
Using the integral Jacobi inequality and recalling the constants $C$ and $c(n)$ from (\ref{J}) we get  
\begin{align}
  \int_{B_2}\phi^2|\nabla_gb|^2dv_g\leq \frac{1}{c(n)}[\int_{B_2}\phi^2\Delta_gbdv_g+\int_{B_2}\phi^2C   dv_g] \nonumber \\
    =-\frac{1}{c(n)}[\int_{B_2}\langle2\phi\nabla_g\phi,\nabla_gb\rangle dv_g+\int_{B_2}\phi^2Cdv_g] \nonumber\\
    \leq \frac{1}{2}\int_{B_2}\phi^2|\nabla_gb|^2dv_g+\frac{2}{c(n)^2}\int_{B_2}|\nabla_g\phi|^2dv_g+\frac{1}{c(n)}\int_{B_2}\phi^2Cdv_g \nonumber\\
    \implies \int_{B_2}\phi^2|\nabla_gb|^2dv_g\leq
    \frac{4}{c(n)^2}\int_{B_2}|\nabla_g\phi|^2dv_g+\frac{2}{c(n)}\int_{B_2}\phi^2Cdv_g. \label{ii}
    \end{align}
    Again using (\ref{b}) and plugging the above inequality in (\ref{III})  we get the following 
\begin{align}
b(0)\leq 
  C(n,\psi, \phi,\delta)[\int_{B_2}b\sum_{i=1}^{n}\frac{1}{1+\lambda_i^2}\sqrt{\det g}dx+\int_{B_2}\sqrt{\det g}dx
  ].\label{pp2}
\end{align}

Next we choose a suitable test function $\eta$ such that on using the Sobolev inequality we get the following estimate 
\begin{align*}
    \int_{B_{2}}\sqrt{\det g}dx\leq C(n)[\int_{B_{2}}\eta ^{\frac{2n}{n-2}}dv_g]^{\frac{n-2}{n}}\\
    \leq C(n)[\int_{B_{2}} |\nabla_g\eta|^2dv_g+\int_{B_{2}}|\eta\nabla_g\psi|^2dv_g]\\
    \leq C(n)[\int_{B_{2}} \sum_{i}^{n}\frac{1}{1+\lambda_i^2}\sqrt{\det g}dx+  \int_{B_{2}}|\eta\nabla_g\psi|^2dv_g   ].
\end{align*}
Using (\ref{b}) we get
\begin{align*}
\int_{B_{2}}\sqrt{\det g}dx\leq C(n,\psi,\eta)\int_{B_{2}}b \sum_{i}^{n}\frac{1}{1+\lambda_i^2}\sqrt{\det g}dx .
\end{align*}
On combining everything and plugging in (\ref{pp2}), we get
\begin{align}
b(0)\leq C(n,\psi,\delta)\int_{B_{2}} b\sum_{i}^{n}\frac{1}{1+\lambda_i^2}\sqrt{\det g}dx.\label{b12}
\end{align}

\item[Step 3.] Now we estimate  $\int_{B_2}b\sum_{i=1}^{n}\frac{1}{1+\lambda_i^2}\sqrt{\det g}dx$.\\
We denote $V=\sqrt{\det g}$. We see that (ref: \cite [(3.2)]{WaY}) on differentiating the complex identity 
\[\ln V+i\sum_{i=1}^n\arctan\lambda_i=\ln \bigg[ \sum_{0\leq 2k\leq n} (-1)^k\sigma_{2k} +i\sum_{1\leq2k+1\leq n}(-1)^k \sigma_{2k+1}\bigg]
\] we get
\[ \bigg( \frac{1}{1+\lambda_1^2},...,\frac{1}{1+\lambda_n^2}\bigg)V=\bigg( \frac{\partial \Sigma}{\partial \lambda_1},...,\frac{\partial \Sigma}{\partial \lambda_n}\bigg)
\]
where \begin{align*}
\sum=\cos\psi \sum_{1\leq 2k+1\leq n}(-1)^k\sigma_{2k+1}-\sin\psi \sum_{0\leq 2k\leq n}(-1)^k\sigma_{2k}
\end{align*} 
where $\psi$ is the Lagrangian phase. On taking the trace, we get
\begin{align*}
    \sum_{i=1}^n \frac{1}{1+\lambda_i^2}V=\sum_{i=1}^n\frac{\partial \Sigma}{\partial \lambda_i}\\
    =\cos\psi \sum _{1\leq 2k+1\leq n}(-1)^k(n-2k)\sigma_{2k}-\sin\psi\sum_{0\leq2k\leq n}(-1)^k(n-2k+1)\sigma_{2k-1}\\
    =c_0(x)+c_{1}(x)\sigma_1+...+c_{n-1}(x)\sigma_{n-1}
\end{align*}
where the variable coefficient $c_i$ now depends on $i,n,\psi$ for all $1\leq i\leq n-1$.
Hence, (\ref{b12}) becomes
\begin{align*}
    \int_{B_2}b\sum_{i=1}^{n}\frac{1}{1+\lambda_i^2}\sqrt{\det g}dx\leq C(n,\psi,\delta)\int_{B_2}b(c_0(x)+c_1(x)\sigma_1+...+c_{n-1}(x)\sigma_{n-1})dx.
    \end{align*}\\

\item[Step 4.] We next estimate the integrals $\int b\sigma_k dx$ for $1\leq k\leq n-1$ inductively, using the divergence structure of $\sigma_k(D^2u)$. Let $L_{\sigma_k}$ denote the matrix $(\frac{\partial \sigma_k}{\partial u_{ij}})$, then we see that 
\begin{align*}
   c_k k\sigma_k(D^2u)=c_k\sum_{i,j=1}^n\frac{\partial \sigma_k}{\partial u_{ij}}\frac{\partial^ 2u}{\partial x_i \partial x_j}=c_k\sum_{i,j=1}^n\frac{\partial}{\partial x_i}\left[ \frac{\partial \sigma_k}{\partial u_{ij}}\frac{\partial u}{\partial x_j}\right]\\
   =c_kdiv(L_{\sigma_k}Du)=div(c_kL_{\sigma_k}Du)-Dc_k\cdot L_{\sigma_k}Du.
\end{align*}
Let $v$ be a smooth cut-off function on $B_{r+1}$ such that $v=1$ on $B_r$, $0\leq v\leq 1$ and $|Dv|<2$. Note that by Lemma \ref{y1}, we have $\sigma_k>0$ and from (\ref{b}) we have $b>\ln\sqrt{4/3}$, which together imply $c_k>0$. So we have
\begin{align}
    \int_{B_r}c_kb\sigma_kdx \leq \int_{B_{r+1}}c_kvb\sigma_k dx \nonumber\\
    =\int_{B_{r+1}}vb\frac{1}{k} [div(c_kL_{\sigma_k}Du)-Dc_k\cdot L_{\sigma_k}Du]dx \nonumber\\
    =\frac{1}{k}\int_{B_{r+1}}[-\langle bDv+vDb, c_k L_{\sigma_k}Du\rangle -vbDc_k\cdot L_{\sigma_k}Du]dx \nonumber \\
    \leq C(n,\psi,\phi)||Du||_{L^{\infty}(B_{r+1})}\{ \int_{B_{r+1}} b\sigma_{k-1}dx+\nonumber \\
    \int_{B_{r+1}}(|\nabla_gb|^2+tr(g^{ij}))\sqrt{\det g}dx \} \label{imp}
\end{align}
where the last inequality follows from the argument used in \cite[(3.6)]{WaY}.\\
We use the integral Jacobi inequality (\ref{IJ}) to simplify the last integral. 
Combining (\ref{ii}) with (\ref{pp2}) we have 
\[\int_{B_2}\phi^2|\nabla_gb|^2dv_g\leq C(n,\psi,\delta)\int_{B_{r+2}}tr(g^{ij})\sqrt{\det g}dx.\]
On rearranging constants, (\ref{imp}) reduces to the following inductive inequality 
\[\int_{B_r}c_kb\sigma_kdx \leq C(n,\psi,\delta)||Du||_{L^{\infty}(B_{r+1})}\{ \int_{B_{r+1}} b\sigma_{k-1}dx+\int_{B_{r+2}}tr(g^{ij})\sqrt{\det g}dx\}.
\]

Now applying the argument used in \cite[(3.7)]{WaY} and the trace-conformality identity we conclude that 
\[b(0)\leq C(n,\psi,\delta)\left[ ||Du||_{L^{\infty}(B_{2n+1})}+||Du||^{2n-2}_{L^{\infty}(B_{2n+1})}\right].
\]
On exponentiating we get 
\[ |D^2 u(0)|\leq C_1\exp[C_2||Du||^{2n-2}_{L^{\infty}(B_{2n+1})}]
\] 
where $C_1$ and $C_2$ are positive constants depending on $||\psi||_{C^{1,1}}$, $n$, and $\delta$.

\end{itemize}

Next, we consider the case $n=2$: we may fix  $\arctan\lambda_3=\pi/2-\delta/2$ and add $(\pi/2-\delta/2)$ to both sides of the two dimensional supercritical equation (\ref{s}) to get:
\[\arctan\lambda_1+\arctan\lambda_2+\arctan\lambda_3=\psi(x)+\frac{\pi}{2}-\frac{\delta}{2}\geq\frac{\pi}{2}+\frac{\delta}{2}.
\]
This again brings us to the three dimensional supercritical equation (\ref{s}) for which Hessian estimates hold good by the above proof.

This completes the proof of Theorem \ref{main1}.

\end{proof}
\textbf{Acknowledgments.} The author is grateful to Yu Yuan for his guidance, support, and many useful discussions. The author is grateful to Ravi Shankar and Micah Warren for helpful comments.

\bibliographystyle{unsrt}
\bibliography{name}

\end{document}